\documentclass{article}
\usepackage[round]{natbib}
\usepackage{amsmath,amsfonts,amssymb,amsthm}
\usepackage{graphicx}
\usepackage{paralist}
\usepackage{booktabs}
\usepackage{url}

\renewcommand{\le}{\leqslant}
\renewcommand{\ge}{\geqslant}

\newcommand{\dnorm}{\mathcal{N}}

\newcommand{\real}{\mathbb{R}}
\newcommand{\bsx}{\boldsymbol{x}}
\newcommand{\bsz}{\boldsymbol{z}}

\newcommand{\rd}{\,\mathrm{d}}

\newcommand{\e}{\mathbb{E}}

\newcommand{\simiid}{\stackrel{\mathrm{iid}}{\sim}}
\newcommand{\cov}{\mathrm{cov}}

\author{Art B. Owen\\Stanford University}
\title{Zero variance self-normalized importance sampling via estimating equations}
\date{September 2025}
\begin{document}
\maketitle
\begin{abstract}
In ordinary importance sampling with a non-negative integrand there exists
an importance sampling strategy with zero variance. Practical sampling
strategies are often based on approximating that optimal solution,
potentially approaching zero variance.
There is a positivisation extension of that method to handle integrands that
take both positive and negative values. 
Self-normalized importance sampling uses a ratio estimate, 
for which the optimal sampler does not have zero variance and so
zero variance cannot even be approached in practice.
Strategies that separately estimate the numerator and denominator
of that ratio can approach zero variance.
This paper develops another zero variance
solution for self-normalized importance sampling.  The first step is to
write the desired expectation as the zero of an estimating equation using
Fieller's technique. Then we apply the positivisation strategy to the
estimating equation. This paper give conditions for existence and uniqueness
of the sample solution to the estimating equation. Then it give conditions
for consistency and asymptotic normality and an expression for the
asymptotic variance.
The sample size multiplied by the variance of the asymptotic formula
becomes arbitrarily close to zero for certain sampling strategies.
\end{abstract}

\newcommand{\pu}{p_u}
\newcommand{\qu}{q_u}

\newcommand{\var}{\mathrm{var}}

\newcommand{\np}{{n_+}}
\newcommand{\nm}{{n_-}}
\newcommand{\npm}{{n_\pm}}
\newcommand{\xip}{\bsx_{i+}}
\newcommand{\xim}{\bsx_{i-}}
\newcommand{\qp}{q_+}
\newcommand{\qm}{q_-}

\newcommand{\of}{\overline{f}}
\newcommand{\uf}{\underline{f}}

\newcommand{\psin}{\Psi_{\np,\nm}}
\newcommand{\dpsin}{\dot\Psi_{\np,\nm}}

\newcommand{\tod}{\stackrel{\mathrm{d\,\,}}{\to}}
\newcommand{\mc}{\mathrm{MC}}

\newtheorem{proposition}{Proposition}
\newtheorem{theorem}{Theorem}

\section{Introduction}

There are two well-known versions of importance sampling,
for estimation of $\e_p(f(\bsx))$ when $\bsx$ has probability
density function $p$.  As described below
they are ordinary importance sampling (OIS) and self-normalized importance
sampling (SNIS).  In OIS, the optimal sampling
density $q$ is well-known to be proportional to $|f(\bsx)|p(\bsx)$. 
When $f$ is nonnegative, OIS with
this distribution has variance zero.  

It is quite unlikely that we could sample from that zero variance distribution.
Even if we could, the OIS computation would require the use
of the very expectation we seek.  The main value of this optimality 
result is that it provides
a guide for choosing an importance sampler.  Near proportionality
to $|f(\bsx)|p(\bsx)$ is a criterion to strive for as part of a variance
reduction strategy.
There is generally no assurance that we can get to zero variance
because our menu of sampling distributions $q$ may not
be rich enough.
In computer graphics, \cite{koll:kell:2002}
call this the problem of `insufficient techniques'.
Adaptive importance sampling (AIS), choosing $q$ from  within a flexible family of densities
has the potential to approach the optimal density, driving the sampling variance
arbitrary close to zero, though the cost of adaptation limits how close 
to zero one could afford to get.  For a survey of AIS, see \cite{buga:etal:2017}.

The situation for SNIS is different.  There, the optimal sampling density 
is $q(\bsx)\propto|f(\bsx)-\e_p(f(\bsx))|$. See Chapter 2 of \cite{hest:1988}.  
Even if we could sample
from this distribution, the result would not have zero variance,
even for nonnegative $f$. That
produces a fundamental limit on the accuracy of SNIS that is not
present in OIS for nonnegative $f$.  It is not just a problem of
insufficient techniques or computational complexity. Even an ideal
AIS strategy would be subject to this lower bound.

When $\Pr_p(f(\bsx)>0)$ and $\Pr_p(f(\bsx)<0)$ are both positive
then there still exists an OIS strategy with zero variance.
\cite{owen:zhou:2000} use two importance samplers,
one for $f_+(\bsx)=\max(f(\bsx),0)$ and one for $f_-(\bsx)=\max(-f(\bsx),0)$.
There are zero variance samplers for both of those and
then because $f(\bsx) = f_+(\bsx)-f_-(\bsx)$ a zero variance sampling 
strategy exists for $\e_p(f(\bsx))$.

This paper develops a new zero variance strategy for SNIS. 
In SNIS, we estimate $\e_p(f(\bsx))$ as a ratio estimator.
The strategy here is to first rewrite the SNIS ratio estimate 
via an estimating equation as in \cite{fiel:1954}.
Then we apply the positivisation strategy from \cite{owen:zhou:2000}
to the estimating equation.  This is done using two OIS samplers
making it possible to approach zero variance.  

Section~\ref{sec:notation} introduces the notation behind the
description above,  presents the customary OIS and SNIS estimators
and then develops the estimating equation approach
where $\hat\mu$ is the estimate of $\mu_0=\e_p(f(\bsx))$.
Section~\ref{sec:properties} gives conditions under which
$\hat\mu$ exists, is unique and converges in probability to $\mu_0$.
It then gives a central limit theorem with $\sqrt{n}(\hat\mu-\mu_0)/\sigma_0\tod\dnorm(0,1)$
where there exist sampling choices that can make $\sigma^2_0$ arbitrarily close to zero. 
That does not mean that $\lim_{n\to\infty}n\var(\hat\mu)=\sigma^2_0$
because the variance of a good approximation need not be a
good approximation to a variance.
This issue also arises with the delta-method approximations
in SNIS \citep{mcbook}.
The estimating equation approach with the EE-SNIS
estimator of this paper is not the first time
a zero variance strategy has been obtained for SNIS.
Section~\ref{sec:references} describes and compares some other approaches.
What unites the methods is finding an OIS solution to a SNIS problem.
The differences are in what distributions we must approximate in order
to approach zero variance.  The prior proposals all require us to
sample from an arbitrarily good approximation to $p$, which might
be very hard.  On the other hand, the TABI method from 
\cite{rain:etal:2020} can use a more general positivisation
strategy than EE-SNIS does.   All the methods potentially benefit
from a coupling strategy proposed by \cite{bran:elvi:2024}.
Section~\ref{sec:conclusions} gives some conclusions.

\section{Notation}\label{sec:notation}
The estimand is 
$$
\mu_0 = \int f(\bsx)p(\bsx)\rd \bsx
$$
for a probability density function (PDF) $p$ on $\real^d$ and 
a real-valued integrand $f$.  We are reserving the symbol $\mu$
to denote some candidate value of this integral that is not
necessarily equal to $\mu_0$.
Apart from this distinction between $\mu$ and $\mu_0$,
most of the notation is like that in \citet[Chapter 9]{mcbook}.
We assume that
$0 <\sigma^2 = \int (f(\bsx)-\mu_0)^2p(\bsx)\rd\bsx<\infty$.
A plain Monte Carlo (MC) estimate of $\mu_0$ is
$$
\hat\mu_{\mc}=\frac1n\sum_{i=1}^n f(\bsx_i)
$$
for $\bsx_i\simiid p$. This estimator satisfies $\e(\hat\mu_{\mc})=\mu_0$
and $\var(\hat\mu_{\mc}) = \sigma^2/n$.

There are two frequently encountered flaws in the MC estimate that each motivate
a form of importance sampling.  The first such flaw is that $f(\bsx)$ may
have an extremely skewed distribution as for example when $f(\bsx) = 1_{\bsx\in A}$
for a set $A$ with very small $\Pr(\bsx\in A)$.  This motivates 
the OIS described below.
A more severe flaw described later, is that we may have no practical way to sample from $p$.
That motivates SNIS.

\subsection{Ordinary importance sampling}
For rare events or very skewed integrands, we might sample
$\bsx\sim q$ instead, where $q$ is a PDF
that satisfies $q(\bsx)>0$ whenever $f(\bsx)p(\bsx)\ne0$.
This $q$ might sample more frequently in any tiny region
where $f$ varies the most under $p$, giving us more
relevant data.
Then we can use the OIS estimate
\begin{align}\label{eq:ois}
\hat\mu_q = \frac1n \sum_{i=1}^n \frac{f(\bsx_i)p(\bsx_i)}{q(\bsx_i)}
\end{align}
for $\bsx_i\simiid q$. We easily see that $\e(\hat\mu_q)=\mu_0$. In other
words multiplying by the ratio $p/q$ adjusts for any bias in sampling from $q$
instead of $p$.

After some algebra we find that
\begin{align}\label{eq:varois}
\var(\hat\mu_q) =\frac1n \int \frac{(f(\bsx)p(\bsx)-\mu_0 q(\bsx))^2}{q(\bsx)}\rd\bsx
=:\frac{\sigma^2_q}n.
\end{align}
When $\mu_0>0$ and $f(\bsx)$ is never negative, the choice $q(\bsx)=f(\bsx)p(\bsx)/\mu_0$ 
yields $\sigma^2_q=0$. We never use this $q$ because the required ratio $p/q$ 
in equation \eqref{eq:ois}  equals the presumably unknown $\mu_0$.
In such cases there exists a sequence $q_\ell$ of densities with $\ell\ge1$
such that $\sigma^2_{q_\ell}\to0$ as $\ell\to\infty$.  We do not necessarily have 
practical tools to construct such a sequence, but at least one exists.
For nonnegative $f$ we do well to have $q$ roughly proportional
to the product $fp$ without also having $q$ be small enough
anywhere to inflate $(fp-\mu_0q)/q$; see the denominator of the integrand in~\eqref{eq:varois}. 
An adaptive sequence of sampling distributions $q_\ell$
may then bring a worthwhile improvement even if it does not converge to zero variance.

\subsection{Self-normalized importance sampling}
The OIS estimator can also be useful in settings where we cannot
sample $\bsx\sim p$, but can instead sample $\bsx\sim q$ for some 
distribution $q$ that is similar to $p$.  This is where a second difficulty 
commonly arises.
In those problems we are often unable to compute $p(\bsx)$ exactly.
For instance, in Bayesian problems,  $p$ may be a posterior
distribution that depends on a quite arbitrary set of data.
We might then only be able to compute an unnormalized density $\pu$ where
$p(\bsx) =\pu(\bsx)/c_p$ for a normalizing constant $c_p = \int \pu(\bsx)\rd\bsx$.
The density $q$ might also be known only up to a normalizing constant $c_q$.
Then we might be able to compute $\qu=q\times c_q$ but not $q$ itself.
Typically, $q$ is a function that we have chosen
from a convenient family of distributions.
In such settings we can often compute $q$.  Indeed being normalized
might be one of our criteria for selecting $q$. Therefore it is quite common
to have a normalized $q$ with an unnormalized $\pu$.

Suppose that $\pu$ is unnormalized and that we can sample from $q$.
For the self-normalized importance sampler below, we can work with either $q$ or $\qu$
so for the next step we just consider $q$ to be $\qu$ with $c_q=1$.
We require  $\qu(\bsx)>0$ whenever $\pu(\bsx)>0$.
Then the SNIS estimator is
\begin{align*}
\tilde \mu_q 
&= \frac1n\sum_{i=1}^n \frac{f(\bsx_i)\pu(\bsx_i)}{\qu(\bsx_i)}
\biggm/\frac1n\sum_{i=1}^n \frac{\pu(\bsx_i)}{\qu(\bsx_i)}\\
&= \frac1n\sum_{i=1}^n \frac{f(\bsx_i)p(\bsx_i)}{q(\bsx_i)}
\biggm/\frac1n\sum_{i=1}^n \frac{p(\bsx_i)}{q(\bsx_i)}.
\end{align*}
The normalizing constants $c_p$ and $c_q$ cancel out
between the numerator and denominator above.
We can use this estimate whether or not $p$ is normalized
and whether or not $q$ is normalized.
We get the same estimate whether we use $q$ or $\qu$,
so we don't need to distinguish $\tilde\mu_q$ from 
an alternative estimate $\tilde \mu_{q_u}$.

The numerator in the second version of $\tilde \mu_q$ above
converges to $\mu$ by the strong law of large numbers while
the denominator similarly converges to $1$. 
Therefore $\Pr(|\tilde\mu_q-\mu_0|>\epsilon)\to0$
as $n\to\infty$ for any $\epsilon>0$.
 We must have $q(\bsx)>0$
whenever $p(\bsx)>0$, whether or not $f(\bsx)=0$, in order to get
convergence in the denominator. The ratio estimate above is generally biased
but the bias typically becomes negligible as $n\to\infty$.

The delta method makes a linear approximation to $\tilde\mu_q$.  The variance
of that linear approximation is denoted $\var_\delta(\tilde\mu_q)$. It
satisfies
$$
\lim_{n\to\infty} n\var_\delta(\tilde\mu_q) = \e_q\Bigl( \frac{p(\bsx)^2}{q(\bsx)^2}(f(\bsx)-\mu_0)^2\Bigr)
=:\tau^2_q.
$$
As noted above, the best density $q$ is proportional to $p(\bsx)|f(\bsx)-\mu_0|$.
If $\Pr_p(f(\bsx)=\mu_0)<1$, then $\tau^2_q>0$.
In the trivial case that $\Pr_p(f(\bsx)=\mu_0)=1$, $\var_p(f(\bsx))=0$
and there does not even exist a density proportional to $p|f-\mu_0|$.
As a result, SNIS never has a zero variance sampler.

\subsection{The positivisation trick}
A similar problem arises for OIS when $f$ takes both positive and
negative values.  Then no single importance sampling estimator
can have zero variance. There is a positivisation strategy in
\cite{owen:zhou:2000} that allows one to approach zero variance using two
importance sampling estimates.
Define $f_+(\bsx) =\max( f(\bsx),0)$ and $f_-(\bsx) = \max(-f(\bsx),0)$.
These are the positive and negative parts of $f$, respectively,
and of course $f(\bsx) = f_+(\bsx)-f_-(\bsx)$.
Let $q_+$ and $q_-$ be normalized densities that we can sample from.
Then the positivised OIS estimate (POIS) of $\mu_0$ is
$$
\frac1{n_+}\sum_{i=1}^{n_+}\frac{f_+(\bsx_{i+})p(\bsx_{i+})}{q_+(\bsx_{i+})}
-\frac1{n_-}\sum_{i=1}^{n_-}\frac{f_-(\bsx_{i-})p(\bsx_{i-})}{q_-(\bsx_{i})},
$$
where $\bsx_{i+}\simiid \qp$ independently of $\bsx_{i-}\simiid\qm$.
It is then possible to have a zero variance estimator
by using OIS separately on positive and negative parts of $f$
via $q_\pm\propto f_\pm p$.

The positivisation trick can be extended.
We can write 
$$\mu_0= \e_p\bigl( (f(\bsx)-c)_+\bigr)-\e_p\bigl( (f(\bsx)-c)_-\bigr)+c$$
for any constant $c$ and use importance sampling estimates
of the two expectations above.  More generally for $g(\bsx)$
with $\e_p(g(\bsx))=\theta$ known 
$$\mu_0= \e_p\bigl( (f(\bsx)-g(\bsx))_+\bigr)-\e_p\bigl( (f(\bsx)-g(\bsx))_-\bigr)+\theta,$$
and we can seek importance sampling estimates of the two expectations above.
We call this generalized POIS (GPOIS).
\cite{owen:zhou:2000} have an example where $f(\bsx)=g(\bsx)$ with very
high probability allowing an importance sampler to focus on the set where they differ.

\subsection{The Fieller trick}

Our goal is to extend the positivisation method to SNIS.
First, we use the Fieller trick \citep{fiel:1954} to write $\mu_0$ as the solution  $\mu$ of 
$$
\e_q\Bigl( \frac{(f(\bsx)-\mu)\pu(\bsx)}{q(\bsx)}\Bigr)=0.
$$
Using positive and negative parts we may write
$$
\e_q\Bigl( \frac{(f(\bsx)-\mu)_+\pu(\bsx)}{q(\bsx)}\Bigr)
-\e_q\Bigl( \frac{(f(\bsx)-\mu)_-\pu(\bsx)}{q(\bsx)}\Bigr)
=0.
$$
For probability density functions $q_\pm$ with
\begin{align}\label{eq:supportq}
q_\pm(\bsx)>0\quad\text{whenever}\quad (f(\bsx)-\mu)_\pm\pu(\bsx)>0
\quad\text{(respectively)},
\end{align}
we can write
\begin{align}\label{eq:defpsi}
\Psi(\mu)&\equiv
\e_{q_+}\biggl( \frac{(f(\bsx)-\mu)_+\pu(\bsx)}{q_+(\bsx)}\biggr)
-\e_{q_-}\biggl( \frac{(f(\bsx)-\mu)_-\pu(\bsx)}{q_-(\bsx)}\biggr).
\end{align}
Then
\begin{align}\label{eq:psiresult}
\Psi(\mu)&=\int (f(\bsx)-\mu)\pu(\bsx)\rd\bsx
 = c_p(\mu_0-\mu),
\end{align}
so $\mu_0$ is the unique solution to $\Psi(\mu)=0$.

We will need the support equation~\eqref{eq:supportq} to hold for all $\mu$ in an interval
containing $\mu_0$.
Suppose that \eqref{eq:supportq} does not hold for some value of $\mu$.
Then
\begin{align}\label{eq:psiresult2}
\Psi(\mu)&=\int_{Q_+} (f(\bsx)-\mu)_+\pu(\bsx)\rd\bsx
-\int_{Q_-} (f(\bsx)-\mu)_-\pu(\bsx)\rd\bsx
\end{align}
for $Q_\pm = \{\bsx\mid q_\pm(\bsx)>0\}$.
Then $\Psi(\mu)$ is still well defined for that $\mu$
but it might have a zero that isn't equal to $\mu_0$.

Both $q_\pm$ above are normalized.
If instead we use $q_{u\pm}$ where $q_\pm = q_{u\pm}/c_{q\pm}$
and the support condition~\eqref{eq:supportq} holds,
then  $\mu_0$ is the unique zero of 
$$
\e_{q_+}\Bigl( \frac{(f(\bsx)-\mu)_+\pu(\bsx)}{q_{u+}(\bsx)}\Bigr)
-\e_{q_-}\Bigl( \frac{(f(\bsx)-\mu)_-\pu(\bsx)}{q_{u-}(\bsx)}\Bigr)
\times\frac{c_{q-}}{c_{q+}}.
$$
This means that we can work with unnormalized distributions
$q_{u\pm}$ so long as we know the ratio of their normalizing
constants.  This is similar to the setting in OIS.  If
we knew the ratio $c_p/c_q$,  we could scale an OIS estimate.
It is however unreasonable to expect that a conveniently available 
distribution $q$ and a problem-specific density $p$ will 
have a known ratio of normalizing constants.

Because $q_\pm$ are both under our control, there are settings
where we could reasonably know the ratio $c_{q+}/c_{q-}$
at much lower cost than knowing each of them separately.
For instance, if $q_\pm$ are Gaussian densities over a high
dimensional space, their normalizing constants require computation
of a determinant and that could be expensive.  If they are Gaussians with
the same covariance matrix and different means, then we know that
those determinants are equal and that can give us unnormalized
densities with  $c_{q-}/c_{q+}=1$.
From here on, we assume that $q_\pm$ are normalized. The
extension to a known ratio $c_{q-}/c_{q+}$ is straightforward.

Using the Fieller trick and the positivisation trick,
we define our estimate  $\hat\mu$ as the solution to
$\psin(\mu)=0$ where
\begin{align}\label{eq:defpsin}
\begin{split}
\psin(\mu)&=\frac1{\np}\sum_{i=1}^\np
\frac{(f(\xip)-\mu)_+\pu(\xip)}{q_+(\xip)}
\\&
-\frac1{\nm}\sum_{i=1}^\nm\frac{(f(\xim)-\mu)_-\pu(\xim)}{q_-(\xim)}
\end{split}
\end{align}
for $\xip\simiid q_+$ independently of $\xim\simiid q_-$
using integers $n_\pm\ge1$.
Because $\hat\mu$ satisfies the estimating equation~\eqref{eq:defpsin} we call it
the estimating equation SNIS, or EE-SNIS.

The value $\psin(\mu)$ is finite for any $\mu\in\real$ 
with probability one, whether or not equation~\eqref{eq:supportq} holds.
The expected value of the first term in it is the nonnegative value
$$\int_{Q_+} (f(\bsx)-\mu)_+\pu(\bsx)\rd \bsx
\le\int (f(\bsx)-\mu)_+\pu(\bsx)\rd \bsx$$ which is finite
whenever $\e_p(|f(\bsx)|)<\infty$.  A quantity with a finite mean
cannot have an unbiased estimate that is infinite with positive probability.
A similar argument shows that the second term is also finite with probability one.

In the next section, we give conditions for existence, uniqueness and consistency
of $\hat\mu$ and an expression for its asymptotic variance. 

\section{Properties of $\hat\mu$}\label{sec:properties}

We take the points $\bsx_{i+}$ for $i=1,\dots,\np$  IID from $\qp$
and independent of $\bsx_{i-}$ for $i=1,\dots,\nm$ that are IID from $\qm$.
We use  $S_+=\{\bsx_{i+}\mid 1\le i\le \np\}$,
$S_-=\{\bsx_{i-}\mid 1\le i\le \nm\}$ and $S=S_+\cup S_-$
to denote sets of sample points that we need to consider.

The function $\psin$ is a continuous nonincreasing piece-wise
linear function on $\real$.  
For $\mu\not\in S$, the derivative of $\psin$ is
\begin{align}\label{eq:defdpsin}
\dpsin(\mu) 
=-\frac{1}{\np}\sum_{i=1}^\np
1_{f(\xip)>\mu}\frac{\pu(\xip)}{q_+(\xip)}
-\frac1{\nm}\sum_{i=1}^\nm1_{f(\xim)<\mu}
\frac{\pu(\xim)}{q_-(\xim)}.
\end{align}
So $\psin$ is differentiable almost everywhere.

\subsection{Existence and uniqueness}
\begin{proposition}
For $n_\pm\ge1$ let $\psin$ be defined by equation~\eqref{eq:defpsin}.
Assume that $\max_{\bsx\in S_+}\pu(\bsx)>0$
and that $\max_{\bsx\in S_-}\pu(\bsx)>0$.
Then there exists a value $\hat\mu$ with $\psin(\hat\mu)=0$.
Let $\overline f =\max\{f(\bsx)\in S_+\mid \pu(\bsx)>0\}$
and $\underline f =\min\{f(\bsx)\in S_-\mid \pu(\bsx)>0\}$.
If $\overline f>\underline f$, then there is at most one solution
to $\psin(\mu)=0$.
\end{proposition}
\begin{proof}
For any $\mu<\min\{f(\bsx)\mid\bsx\in S\}$
$$
\psin(\mu) = 
\frac1{\np}\sum_{i=1}^\np\frac{(f(\xim)-\mu)\pu(\xim)}{q_-(\xim)}>0
$$
by our assumption on $\max_{\bsx\in S_+}\pu(\bsx)>0$.
Similarly $\psin(\mu)<0$ for any $\mu>\max\{f(\bsx)\mid\bsx\in S\}$.
Because $\psin$ is a continuous function taking at least one positive
value and at least one negative value, it has at least one zero.

If $\mu<\overline{f}$, then the first term in~\eqref{eq:defdpsin}
is strictly negative. If $\mu>\underline{f}$, then the second term
there is strictly negative. If $\overline{f} >\underline{f}$, then~\eqref{eq:defdpsin}
is strictly negative for all $\mu\in\real$. 
\end{proof}

For many applications $\pu(\bsx)$ will be positive
at every $\bsx\in S$.  Some applications may have `holes'
meaning sets of $\bsx$ values
where one or both of $q_\pm(\bsx)>0$ but $\pu(\bsx)=0$. 
For eventual existence of $\hat\mu$ we only require that the set of
holes has probability below one under each of $q_\pm$.
When we design $q_+$ we will want to oversample regions where $f(\bsx)>\mu_0$
and similarly for $q_-$ we will want to oversample regions where $f(\bsx)<\mu_0$.
Then $\overline{f} >\underline{f}$ will be usual.  It will be usual for $q_\pm$ to
overlap and then $\Pr(\overline{f} \le \underline{f})$ will tend to zero
exponentially fast in $\min(n_+,n_-)$.
We did not need the support condition of equation~\eqref{eq:supportq} to hold for all $\mu\in\real$.
From here on, we suppose that $\hat\mu$ exists and is unique.

\subsection{Consistency}

Here we first show that $\psin(\mu)$ converges
to $\Psi(\mu)$.  Then we use monotonicity
of $\psin$ to show that $\hat\mu\tod\mu_0$.

\begin{proposition}\label{prop:locallln}
Let $q_\pm$ satisfy equation~\eqref{eq:supportq}
for some $\mu\in\real$. Then
$$
\lim_{\min(\np,\nm)\to0}\Pr\bigl( |\psin(\mu)-\Psi(\mu)|>\epsilon\bigr)=0.
$$
\end{proposition}
\begin{proof}
This follows from the law of large numbers applied
to each term in~\eqref{eq:defpsin}.
\end{proof}

\begin{proposition}\label{prop:consistency}
Let $\np(n)$ and $\nm(n)$ be nondecreasing functions of a
positive integer index $n$ with $\min(\np,\nm)\to\infty$ as $n\to\infty$.
Let $q_\pm$ be probability density functions that are
positive whenever $(f(\bsx)-\mu)_\pm\pu(\bsx)>0$ with
this positivity holding simultaneously for all $\mu\in(\mu_0-\eta,\mu_0+\eta)$
for some $\eta>0$.
Assume that $\psin$ has a unique zero $\hat\mu$ for all sufficiently large $n$.
Then for any $\epsilon>0$,
\begin{align}\label{eq:consistency}
\lim_{n\to\infty}\Pr\bigl( |\hat\mu-\mu_0|>\epsilon\bigr)\to0.
\end{align}
\end{proposition}
\begin{proof}
We adapt the argument in Lemma 5.10 of \cite{vand:1998}.
The function $\psin(\cdot)$ is continuous and nonincreasing
on all of $\real$,
and for large enough $n$, it has a unique zero $\hat\mu$.
Without loss of generality we can assume that $\epsilon<\eta$.
Then for large enough $n$,
$$
\Pr\bigl( |\hat\mu-\mu_0|>\epsilon\bigr)
\le\Pr\bigl( \psin(\mu_0-\epsilon)<0\bigr)
+\Pr\bigl( \psin(\mu_0+\epsilon)>0\bigr).
$$
By Proposition~\ref{prop:locallln},
$\psin(\mu_0-\epsilon)$ converges in probability to 
$\Psi(\mu_0-\epsilon)=c_p\epsilon$
and $\psin(\mu_0+\epsilon)$ converges in probability
 to $-c_p\epsilon$ establishing~\eqref{eq:consistency}.
\end{proof}

\subsection{Asymptotic variance and central limit theorem}
Now we turn to the variance of $\hat\mu$.  
The estimate $\hat\mu$ satisfies
\begin{align*}
0=\psin(\hat\mu) 
&=\psin(\mu_0) +(\hat\mu-\mu_0) \int_0^1 \dpsin( \mu_0+t(\hat\mu-\mu_0))\rd t
\end{align*}
and so
\begin{align}\label{eq:muhatexpr}
\hat\mu = \mu_0 - \frac{\psin(\mu_0)}{\int_0^1 \dpsin( \mu_0+t(\hat\mu-\mu_0))\rd t}.
\end{align}
The denominator is an average of $\dpsin$ over the asymptotically
small interval
from $\hat\mu$ to $\mu_0$. For all but a finite number of
values $\mu$ in this interval,  $\dpsin(\mu)$ exists
and approaches $\dot\Psi(\mu)=-c_p$.
Under appropriate conditions a uniform law of large numbers
will make that integral converge in probability to $-c_p$.

\begin{proposition}\label{prop:ulln}
Under the conditions in Proposition~\ref{prop:consistency},
$$
\lim_{n\to\infty} \Pr\biggl(\sup_{|\mu-\mu_0|<\eta,\,\mu\not\in S}
|\dpsin(\mu) -\dot\Psi(\mu)|>\epsilon\biggr)=0
$$
for any $\eta>0$ and any $\epsilon>0$.
\end{proposition}
\begin{proof}
  The set $S$ that we exclude depends on $n$, but the
  union over all $n$ is still countable, hence of measure
  zero. In this proof, we consider $\mu$ not in $S$ for any $n$.
Minus the second average in~\eqref{eq:defdpsin} for $\dpsin(\mu)$ takes the form
$$
\hat F_-(\mu)=\frac1{n_-}\sum_{i=1}^{n-}1_{ Y_i <\mu}W_i
$$
for $Y_i = f(\bsx_{i-})$ and a nonnegative weight $W_i = \pu(\bsx_{i-})/\qm(\bsx_{i-})$
that has a finite expectation. It is a weighted version of an empirical CDF (apart from
using $Y_i<\mu$ instead of $Y_i\le\mu$). We see below that it converges uniformly in $\mu\in(\mu_0-\eta,\mu_0+\eta)$ to its expectation by a generalization of
the Glivenko-Cantelli theorem. The same holds for the first average.

To generalize the Glivenko-Cantelli theorem, let $F_-(\mu) = \e_{\qm}(1_{Y<\mu} W)$.  Choose some bound $M<\infty$
and write $\bar W = \min(W,M)$,  $\bar W_i = \min(W_i,M)$,
$F_-(\mu;M) = \e_{\qm}(1_{Y<\mu}\bar W)$,
and $\hat F_-(\mu;M) =(1/\nm)\sum_{i=1}^{\nm}1_{Y_i<\mu}\bar W_i$.
Then
\begin{align*}
\sup_{\mu} |\hat F_-(\mu)-F_-(\mu)| &\le 
\sup_\mu |\hat F_-(\mu;M)-F_-(\mu;M)| \\
&+\frac1{\nm}\sum_{i=1}^{\nm}(W_i-\bar W_i)+\e_{\qm}(W-\bar M).
\end{align*}
Here and below the supremum is taken over $\mu\in (\mu_0-\eta,\mu_0+\eta)$.

The third term can be made smaller than any $\epsilon_3>0$ by
taking $M$ large enough.  The second term can be made smaller than
any $\epsilon_2>0$ with probability close to 1 by taking $\nm$ large enough.
It remains to show that the first term is uniformly bounded by $\epsilon_3$
with probability close to $1$ for $\nm$ large enough. After that
we choose $\epsilon_1+\epsilon_2+\epsilon_3<\epsilon/2$
and then make a comparable analysis of the first average in \eqref{eq:defdpsin}.

Because $F_-$ is nondecreasing and bounded between $0$ and $M$,
for any $\delta>0$, we may choose a fine grid of values $\nu_0<\nu_1<\dots<\nu_L$
with $\nu_0=\mu_0-\eta$ and $\nu_L=\mu_0+\eta$
so that $\Delta_\ell=F_-(\nu_\ell)-F_-(\nu_{\ell-1})=
\e_{\qm}(\bar W1_{\nu_{\ell-1}<Y\le\nu_\ell})<\delta_1$
for all $\ell=1,\dots,L$.  

Then $\Pr( \max_{0\le\ell\le L}|\hat F_-(\nu_\ell)-F_-(\nu_\ell)|>\delta_2)$
can be made as small as we like by choosing $\nm$ large enough.
We then take $\delta_1+\delta_2<\epsilon_3$.
\end{proof}

The averages in our estimating equations are OIS estimates
$$
\hat\mu_\pm 
= \frac1{n_\pm}\sum_{i=1}^{n_\pm} \frac{(f(\bsx_{i\pm})-\mu_0)_\pm \pu(\bsx_{i\pm})}{q_\pm(\bsx_{i\pm})}
= \frac{1}{n_\pm}\sum_{i=1}^{n_\pm} \frac{c_p(f(\bsx_{i\pm})-\mu_0)_\pm p(\bsx_{i\pm})}{q_\pm(\bsx_{i\pm})},
$$
which are unbiased for
\begin{align}\label{eq:mupm}
\mu_\pm = c_p\int (f(\bsx)-\mu_0)_\pm p(\bsx)\rd\bsx= \int (f(\bsx)-\mu_0)_\pm \pu(\bsx)\rd\bsx.
\end{align}
They have zero variance when $q_\pm \propto (f(\bsx)-\mu_0)_\pm p(\bsx)$.
In general, the variances of these OIS estimates are
\begin{align}\label{eq:varoiszero}
\frac{\sigma^2_\pm}{n_\pm}\quad\text{where}\quad
\sigma^2_\pm=\int\frac{[(f(\bsx)-\mu_0)_\pm \pu(\bsx)-\mu_\pm q_\pm(\bsx)]^2}{q_\pm(\bsx)}\rd\bsx
\end{align}
Now we can state the main result.

\begin{theorem}
Assume the conditions of Proposition~\ref{prop:consistency}
with $n_+/n \to\theta$ and $n_-/n\to(1-\theta)$ for $0<\theta<1$.
If both of $\sigma^2_\pm$ from \eqref{eq:varoiszero}
are finite then for all $t\in\real$,
$$
\lim_{n\to\infty} 
\Pr\left(\frac{\sqrt{n}(\hat\mu-\mu_0)}
{\sqrt{\frac{\sigma^2_+}\theta+\frac{\sigma^2_-}{1-\theta}}}
\le t\right)=\Phi(t)
$$
where $\Phi$ is the $\dnorm(0,1)$ cumulative distribution function.
\end{theorem}
\begin{proof}
First  $\e(\psin(\mu_0))=0$. To find $\var(\psin(\mu_0))$ note that
$$
\frac1{\np}\sum_{i=1}^{\np}\frac{(f(\bsx_{i+})-\mu)_+\pu(\bsx_{i+})}{\qp(\bsx_{i+})}
=\frac{c_p}{\np}\sum_{i=1}^{\np}\frac{(f(\bsx_{i+})-\mu)_+p(\bsx_{i+})}{\qp(\bsx_{i+})}
$$
has variance $c_p^2\sigma^2_+$. Combined with the corresponding result for
the second term in $\psin(\mu_0)$ we get
$$
\var\bigl(\sqrt{n}\psin(\mu_0)\bigr)
=n\times\Bigl( \frac{\sigma^2_+}{\np}+ \frac{\sigma^2_-}{\nm}\Bigr)c_p^2
\to\Bigl( \frac{\sigma^2_+}{\theta}+ \frac{\sigma^2_-}{1-\theta}\Bigr)c_p^2.
$$
Then by the central limit theorem,
$$\frac{\sqrt{n}\psin(\mu_0)}
{\sqrt{\frac{\sigma^2_+}\theta+\frac{\sigma^2_-}{1-\theta}}}
\tod\dnorm(0,c_p^2).$$
Next,
$$
\hat\mu-\mu_0 = -\frac{\psin(\mu_0)}
{\int_0^1 \dpsin( \mu_0+t(\hat\mu-\mu_0))\rd t}
$$
and the denominator converges in probability to $-c_p$
by Proposition~\ref{prop:ulln}, after noticing that
$\dot\Psi(\mu)=-c_p$ for almost all $\mu$. 
Finally by Slutsky's theorem 
$\sqrt{n}(\hat\mu-\mu_0)/(\sigma^2_+/\theta+\sigma^2_-/(1-\theta))^{1/2}\tod\dnorm(0,1)$.
\end{proof}

The best choice for $\theta$ is $\sigma_+/(\sigma_++\sigma_-)$.
While densities $q_\pm(\bsx)\propto (f(\bsx)-\mu_0)_\pm p(\bsx)$ give
a zero variance estimate for $\psin(\mu_0)$.
These densities may fail to satisfy \eqref{eq:supportq} for some
$\mu\ne\mu_0$.  In practice we would want to modify them
so that \eqref{eq:supportq} holds for all $|\mu-\mu_0|<\eta$
for some $\eta$. This support expansion can be done with only
very small changes to the optimal $q_\pm$.

It may pay to use some coupling between the samples $\bsx_{i+}$
and $\bsx_{i-}$ as \cite{bran:elvi:2024} do in their AIS.
That is most easily done taking $n_+=n_-=n$ and using some
joint distribution $q$ for $n$ IID pairs $(\bsx_{i+},\bsx_{i-})$, with
marginal densities $q_+$ and $q_-$.
Then the appropriate variance quantity is
\begin{align}\label{eq:withcov}
\frac1n\biggl(
\sigma_+^2 +\sigma_-^2
-2\cov_q\Bigl( 
\frac{(f(\bsx_+)-\mu_0)_+p(\bsx_+)}{q_+(\bsx_+)},
\frac{(f(\bsx_-)-\mu_0)_-p(\bsx_+)}{q_-(\bsx_-)}
\Bigr) \biggr).
\end{align}
We would want the covariance above to be positive and arranging
for that is problem specific. A general guideline is that
to give the values $(f(\bsx_+)-\mu_0)_+$
and $(f(\bsx_-)-\mu_0)_-$ a positive correlation we would normally
make $f(\bsx_+)$ and $f(\bsx_-)$ negatively associated.  
The ratios $p(\bsx_+)/q_+(\bsx_+)$ and $p(\bsx_-)/q_-(\bsx_-)$
could potentially undo that positive correlation, so we might
also seek to keep those ratios positively associated.

\subsection{Alternative centering}

Let $g(\bsx)$ have known mean $\e_p(g(\bsx))=\theta$,
and write $g_0(\bsx) = g(\bsx)-\theta$.
For OIS, we could write
$$
\mu = \e_p( (f(\bsx)-g(\bsx))_+)
-\e_p( (f(\bsx)-g(\bsx))_-)+\theta
$$
and then there exist arbitrarily accurate
estimators of the above two expectations.
For EE-SNIS we can do this for $g_0$ but not generally for $g$.
By writing
\begin{align*}
0 
&= \e_q\biggl( \frac{(f(\bsx)-g(\bsx)-\mu)\pu(\bsx)}{q(\bsx)}\biggr)+c_p\theta\\
&= \e_{\qp}\biggl( \frac{(f(\bsx)-g(\bsx)-\mu)_+\pu(\bsx)}{\qp(\bsx)}\biggr)
-\e_{\qm}\biggl( \frac{(f(\bsx)-g(\bsx)-\mu)_-\pu(\bsx)}{\qm(\bsx)}\biggr)+c_p\theta
\end{align*}
we see that to center $f$ around $g$, we need to know
$c_p\theta$.  This is automatic when we know that $\theta=0$
but otherwise requires knowledge of $c_p$ that we do not
ordinarily have when we want to use SNIS.

The consequence is that we can replace $f$ by $f-g_0$.
This simply means that we are allowed to choose from among
integrands that are known to have the same expectation
as $f$ when $\bsx\sim p$.

\section{Other zero variance SNIS estimates}\label{sec:references}
The SNIS estimator
is a ratio estimator, where both numerator and denominator are Monte Carlo
estimates of some expectations.  While most papers use the same 
distribution $q$ and sample points $\bsx_i\sim q$
in the  both numerator and denominator, there is no reason that
we have to do this.  We could sample $\bsx_i\simiid q_1$ for $i\in S_1$
independently of $\bsx_i\simiid q_2$ for $i\in S_2$.
Here, and in what follows, $S_j$ for different $j$ are disjoint sets of indices
with cardinailty $|S_j|$.
With these samples we estimate $\mu$ by
\begin{align}\label{eq:doubleest}
\hat\mu_{q_1,q_2}=
\frac1{|S_1|}\sum_{i\in S_1}
\frac{f(\bsx_i)\pu(\bsx_i)}{q_1(\bsx_i)}
\Bigm/
\frac1{|S_2|}\sum_{i\in S_2}
\frac{\pu(\bsx_i)}{q_2(\bsx_i)}.
\end{align}
The numerator and denominator above are both  OIS estimates.
When $f$ is nonnegative there is a zero variance choice  $q_1$
for the numerator.  Taking $q_2=p$ always gives a zero variance
for the denominator.

The estimate $\hat \mu_{q_1,q_2}$ was proposed
in the double proposal importance sampling (DPIS)
estimator of \cite{lamb:etal:2018} and,
independently, in the amortized Monte Carlo integration (AMCI)
procedure of \cite{goli:wood:rain:2019}.
\cite{bran:elvi:2024} note via \cite{hest:1988}
that this device was also used by \cite{goya:held:shah:1987}.
In that paper, $p$ was a discrete time Markov process and they
used different non-Markovian samplers $q_1$ and $q_2$, so
it is easy to see how it could have been missed by authors
working in a more general framework.

When $f$ takes both positive and negative values, then we
can use positivisation and the estimate
\begin{align}\label{eq:tabi1}
\hat\mu_{q_1,q_2,q_3}=
\frac{\frac1{|S_1|}\sum_{i\in S_1}
\frac{f_+(\bsx_i)\pu(\bsx_i)}{q_1(\bsx_i)}
-
\frac1{|S_2|}\sum_{i\in S_2}
\frac{f_-(\bsx_i)\pu(\bsx_i)}{q_2(\bsx_i)}}
{
\frac1{|S_3|}\sum_{i\in S_3}
\frac{\pu(\bsx_i)}{q_3(\bsx_i)}}
\end{align}
where $\bsx_i\simiid q_j$ for $i\in S_j$ (disjoint).
This is the approach taken by \cite{rain:etal:2020}
in their target-aware Bayesian inference (TABI) estimators.
A straightforward modification is to use
\begin{align}\label{eq:tabi2}
\hat\mu_{q_1,q_2,q_3,q_4}=
\frac{\frac1{|S_1|}\sum_{i\in S_1}
\frac{f_+(\bsx_i)\pu(\bsx_i)}{q_1(\bsx_i)}
}
{\frac1{|S_3|}\sum_{i\in S_3}
\frac{\pu(\bsx_i)}{q_3(\bsx_i)}}
-
\frac{\frac1{|S_2|}\sum_{i\in S_2}
\frac{f_-(\bsx_i)\pu(\bsx_i)}{q_2(\bsx_i)}
}{
{\frac1{|S_4|}\sum_{i\in S_4}
\frac{\pu(\bsx_i)}{q_4(\bsx_i)}}
}.
\end{align}
This allows two different approximations $q_3$ and $q_4$
that should both approximate~$p$.  Incorporating $\mu_4$
increases cost and complexity but it gives another way to reduce
variance, when as described below, the samples from various $q_j$
are coupled.

\cite{rain:etal:2020} mention in passing that TABI can be generalized in
the same way that GPOIS generalized POIS. A generalized (GTABI)
version of~\eqref{eq:tabi1} is
\begin{align}\label{eq:gtabi1}
\hat\mu_{q_1,q_2,q_3,g}=
\frac{\frac1{|S_1|}\sum_{i\in S_1}
\frac{(f(\bsx_i)-g(\bsx_i))_+\pu(\bsx_i)}{q_1(\bsx_i)}
-
\frac1{|S_2|}\sum_{i\in S_2}
\frac{(f(\bsx_i)-g(\bsx_I))_-\pu(\bsx_i)}{q_2(\bsx_i)}}
{
\frac1{|S_3|}\sum_{i\in S_3}
\frac{\pu(\bsx_i)}{q_3(\bsx_i)}}
+\theta
\end{align}
where $\theta=\e_p(g(\bsx))$.
We can also generalize~\eqref{eq:tabi2}.

To approach zero variance in the above algorithms
we need a density $q_j$ that approaches $p$.
The original motivation for SNIS is that $p$ is difficult
to sample from. We might then expect it to be particularly
hard to approximate well.  In Bayesian settings, our integrands
of interest commonly include $f(\bsx) =x_j$ or $x_j^2$
when we want posterior means and variance of single parameters.
It is speculative, though plausible, to suppose that in these cases
$f(\bsx)_\pm\pu(\bsx)$  or $f(\bsx)^2\pu(\bsx)$ 
might be easier to approximate than $\pu(\bsx)$ itself because
we know how to sample $\bsx$ in order to make $(x_j)_\pm$ or $x_j^2$ large.

While the methods above can approach zero variance,
they will generally not have zero variance at any
stage of AIS. Then
using independent samples in all the constituent integral
estimates misses an opportunity to arrange for
some error reductions similar to what we can
get from the method of common random numbers.
 \cite{bran:elvi:2024} propose
a generalized self-normalized importance sampler
that builds in an association between the numerator
and denominator sample values.
Let $\bsx_i$ be used
in the numerator and $\bsz_i$ in the denominator. Then they study
\begin{align}\label{eq:crndouble}
\hat\mu_q=
\frac1n\sum_{i=1}^n\frac{f(\bsx_i)\pu(\bsx_i)}{q_1(\bsx_i)}\Bigm/\frac1n\sum_{i=1}^n\frac{\pu(\bsz_i)}{q_2(\bsz_i)}
\end{align}
where $(\bsx_i,\bsz_i)$ are IID draws from a joint distribution 
$q$ where, for each index $i$,  
$\bsx_i\sim q_1$ and $\bsz_i\sim q_2$  can be dependent.
The problem now is to choose the joint distribution $q(\bsx,\bsz)$.
It must have marginal probability densities $q_1(\bsx)$ and $q_2(\bsz)$
though the joint distribution need not have a probability
density function. For example $\Pr_q(\bsx=\bsz)>0$ is allowed,
or some other deterministic relationship between $\bsx$
and $\bsz$ could have positive probability.

For given marginals $q_1$ and $q_2$, \cite{bran:elvi:2024}
define the best joint distribution $q$ for
the estimator in~\eqref{eq:crndouble} 
to be the one that maximizes
\begin{align}\label{eq:howtocouple}
\e_q\Bigl( \frac{q_1^*(\bsx)q_2^*(\bsz)}{q_1(\bsx)q_2(\bsz)}\Bigr)
\end{align}
where $q_1^*\propto fp$ and $q_2^*=p$ are the
two optimal samplers. Finding the best $q$ is
a difficult optimization problem and to get something 
computationally tractable they work within some parametric families.

The main difference between these zero
variance approaches centers around which
quantities we must be able to approximate
by a normalized distribution $q_j$ that  we can sample from.
Table~\ref{tab:todo} summarizes that task for
six different methods.  All of the OIS methods require
a known value of $c_p$. DPIS/AMCI and TABI require an
approximation for $p$.  EE-SNIS does not require
an approximation to~$p$.  On the other hand it does
not work with general centering variable $g(\bsx)$
the way that TABI generalizes to GTABI.

\begin{table}
\centering
\begin{tabular}{lllllll}
\toprule
Method & Targets for $q_j$ & Other needs\\
\midrule
OIS & $|f|p$ & known $c_p$ and $f\ge0$ (or $f\le 0$)\\
POIS & $f_+p$,\quad  $f_-p$ & known $c_p$\\
GPOIS & $(f-g)_+p$ \quad $(f-g)_-p$ & known $c_p$ and $\e_p(g(\bsx))$\\
DPIS/AMCI & $fp$ \quad $p$ \\
TABI & $f_+p$ \quad $f_-p$ \quad $p$\\
GTABI & $(f-g)_+p$ \quad $(f-g)_-p$ \quad $p$ & known $\e_p(g(\bsx))$\\
EE-SNIS & $(f-\mu_0)_+p$\quad $(f-\mu_0)_-p$\\
\bottomrule
\end{tabular}
\caption{\label{tab:todo}
  For each method, the second column lists the
  distributions that we need some $q_j$ to approximate,
  in order to approach a zero variance solution.
  The third column lists other requirements.
}
\end{table}

It is reasonable to suppose that careful coupling of the
distributions $q_\pm$ could bring an improvement
to EE-SNIS.  It would natural
to require $\np=\nm$ and then seek to maximize
the covariance in~\eqref{eq:withcov}.

\section{Conclusions}\label{sec:conclusions}

This paper introduces a new zero variance strategy for
SNIS based on estimating equations.  The EE-SNIS algorithm
requires samplers that approximate different distributions
than prior solutions TABI/DPIS/AMCI require.  Those methods separately
estimate numerator and denominator in the SNIS ratio estimate
while EE-SNIS does not require us to find a sampler that approximates
$p$.  This supports an alternative approach to AIS for settings
where $p$ is difficult to approximate well.  
Devising a specific AIS for EE-SNIS is outside the scope of this article.  
Similarly, determining whether new or old approaches 
lead to a better AIS depends on the families of adaptive samplers in use as well as the  
underlying $p$ and integrand(s) $f$ of interest, and is outside  
the present scope.
Properly addressing either of these two issues would require
significant additional length.

\section*{Acknowledgments}

This work was supported by the U.S.\ National Science Foundation
under grant DMS-2152780.

\bibliographystyle{apalike}
\bibliography{rare}
\end{document}